\documentclass[11pt]{amsart}
\usepackage{graphicx}
\usepackage{pinlabel}
\usepackage{amssymb}
\usepackage{color}

\theoremstyle{plain}
 \newtheorem{thm}{Theorem}[section]
 \newtheorem{cor}[thm]{Corollary}
 
 \newtheorem{lem}[thm]{Lemma}
 \newtheorem{prop}[thm]{Proposition}

 \newtheorem{que}[thm]{Question}
 
 \newtheorem{rem}[thm]{Remark}
 
 \newcommand{\R}{\mathbb R}

\begin{document}

\title{Maximal knotless graphs}

\author{Lindsay Eakins}
\author{Thomas Fleming}
\author{Thomas W.~Mattman}
\address[LE \& TWM]{Department of Mathematics and Statistics,
California State University, Chico,
Chico, CA 95929-0525}
\email{TMattman@CSUChico.edu}


\begin{abstract}
A graph is maximal knotless if it is edge maximal for the property of knotless embedding in $\R^3$. 
We show that such a graph has at least $\frac74 |V|$ edges, and construct an infinite family of maximal
knotless graphs with $|E| < \frac52|V|$. With the exception of $|E| = 22$, we 
show that for any $|E| \geq 20$ there exists a maxmal knotless graph of size $|E|$. We  classify the maximal knotless graphs through nine vertices and 20 edges. We determine which of these maxnik graphs are the clique sum of smaller graphs and construct an infinite family of maxnik graphs that are not clique sums.

\end{abstract}

\maketitle

\section{Introduction}

A graph $G$ is {\em maximal planar} if it is edge maximal for the property of being a planar graph.   That is, $G$ is either a planar complete graph, or else adding any missing edge to $G$ results in a non-planar graph. Maximal planar graphs are triangulations and are characterized by the number of edges: a planar graph with $|V| \geq 3$ is maximal planar if and only if $|E| = 3|V|-6$.
 
Naturally, planarity is not the only property of graphs that that can be studied with respect to edge maximality.  A graph is {\em intrinsically linked} if every embedding of the graph in $\R^3$ contains a non-split link.
Some early results on {\em maximal linkless} (or  {\em maxnil}) graphs--those that are edge maximal for the property of not being intrinsically linked--include a family of maximal linkless graph with $3|V|-3$ edges \cite{J}, and the fact that the graph $Q(13,3)$ is a splitter for intrinsic linking, a property that implies it is maximal linkless \cite{Mh}. Recently there have been several new results including families of maxnil graphs with $3|V|-3$ edges (rediscovering J{\o}rgensen's examples) \cite{DF}, with $\frac{14}{5}|V|$ edges \cite{A}, and with $\frac{25}{12}|V|$ edges \cite{NPP}.  Lower bounds for the number of edges required for a maxnil graph have been established \cite{A}, and methods for creating new maxnil graphs via clique sum have been developed \cite{NPP}.

We extend this work with what appears to be the first study of {\em maximal knotless} graphs.
A graph is {\em intrinsically knotted (IK)} if every embedding in $\R^3$ includes a non-trivially knotted cycle, and a graph is {\em not 
IK} or {\em nIK} if it has a knotless embedding, that is, an embedding in which every cycle is a trivial knot. We will call a 
graph that is edge maximal for the nIK property maximal knotless or {\em maxnik.}

In Section \ref{sec:class}, we establish a connection between maximal 2-apex graphs and maxnik graphs, specifically that a 2-apex graph is maxnik if and only if it is maximally 2-apex. This connection is instrumental in allowing the identification of all maxnik graphs with nine or fewer vertices, and with 20 or fewer edges.
We remark that there is an analogous connection between maximal apex graphs and maxnil graphs that may be of independent interest.

We consider clique sums of maxnik graphs in Section \ref{sec:cliq}, and are able to establish similar, if weaker, results to those of \cite{NPP}.  Most importantly, we show that the edge sum of two maxnik graphs $G_1$ and $G_2$ on an edge $e$ is maxnik if $e$ is non-triangular (i.e., not part of a $3$-cycle)
in at least one $G_i$. Similarly, we provide conditions that ensure that the clique sum over $K_3$ of two maxnik graphs is again maxnik.
These results are used in Section \ref{sec:bounds} to construct new maxnik graphs from those found in Section \ref{sec:class}.

We then turn to studying general properties of maxnik graphs in Section \ref{sec:bounds}.  We establish a lower bound for the number of edges in a maxnik graph of $\frac74|V|$, and construct an infinite family of maxnik graphs with fewer than $\frac52|V|$ edges.  A maximal planar graph has $|E| = 3|V| - 6$, and maximal $k$-apex graphs also have a fixed number of edges depending on $|V|$. In contrast, the number of edges in maxnil and maxnik graphs can vary. We show that, except for
$|E| = 22$, given any $|E| \geq 20$, there exists a maxnik graph of size $|E|$.

We will call a maxnik graph {\em composite} if it is the clique sum of two smaller graphs.  Otherwise we say it is {\em prime}. These terms are analogous to knots, where a knot is composite if is the connected sum of two non-trivial knots, and prime otherwise.
The infinite families of maxnik graphs constructed in Section \ref{sec:bounds} are all composite, as they are clique sums of smaller maxnik graphs.  In Section \ref{sec:prime}, we  classify the maxnik graphs found in Section \ref{sec:class} and construct an infinite family of prime maxnik graphs.

\section{Classification through order nine and size 20}
\label{sec:class}

\begin{thm} 
\label{thm:mind}%
A maxnik graph is $2$-connected. If $|V| \geq 3$, then $\delta(G) \geq 2$.
If $|V| \geq 7$, then $20 \leq |E| \leq 5n-15$.
\end{thm}

\begin{proof}
Suppose $G$ is maxnik. If $G$ is not connected, then, in a knotless embedding, add an edge $e$
to connect two components. This is a knotless embedding of $G+e$, contradicting $G$ being maximal knotless. 

Suppose $G$ has connectivity one with cut vertex $v$. 
Label the two components of $G \setminus v$ as $A$ and $B$.  Let $a$ be a neighbor of $v$ in $A$ and $b$ be a neighbor of $v$ in $B$.  These must exist as $G$ is connected.  We will argue that $G + ab$ is also nIK, a contradiction. 

Form an embedding of $G$ as follows: Embed $A$ and $B$ so that they are knotless and disjoint.  Embed $v$ on a plane separating them.  Embed edges from $v$ to $A$ on the $A$ side of the plane, edges from $v$ to $B$ on the $B$ side, so that the embedding remains knotless.  
Now, isotope the rest of $A$ (and $B$) until edge $va$ (and $vb$) is embedded in the plane.  Next add edge $ab$ so that the triangle $abv$ bounds a disk.

Any cycle contained in $A$ (or in $B$) is an unknot.  Any cycle c that uses vertices from both $A$ and $B$ must use at least
two vertices in the triangle $abv$.  Since $abv$ bounds a disk, this means the cycle $c$ is a connected sum of a cycle in $A$ and a cycle in $B$.  Since those are unknots, $c$ must be as well. This shows $G+ab$ is nIK, contradicting $G$ being maxnik.
So a maxnik graph cannot have connectivity one and must be $2$-connected.

Suppose $G$ is a maxnik graph with $|V| \geq 3$. Since $G$ is connected, $\delta(G) > 0$. If $v \in V(G)$ has degree one, let $u$ be the neighbor of $v$ and $w \neq v$ a different neighbor of $u$. In a knotless embedding of $G$, we can introduce
the edge $vw$ that closely follows the path $v, u, e$. This gives a knotless embedding of $G+vw$, contradicting 
the maximality of $G$.

Suppose $G$ is maxnik with $|V| \geq 7$. The lower bound on size is a consequence of the observation \cite{JKM,M} that 
an IK graph has at least 21 edges. The upper bound follows as a graph with $|E| \geq 5|V| - 14$ has a $K_7$ minor
and is therefore IK \cite{Md, CMOPRW}.
\end{proof}

In Theorem~\ref{thm:NPP5} below, we construct an infinite family of maxnik graphs, each with $\delta(G) = 2$. 

\begin{thm} 
\label{thm:twoapex}%
A $2$-apex graph is maxnik if and only if it is maximal $2$-apex.
\end{thm}

\begin{proof}
Let $G$ be $2$-apex. If $G$ is not maximal $2$-apex, then there is an edge $e$ so that $G+e$ is $2$-apex, hence nIK \cite{BBFFHL,OT}. This 
shows that $G$ is not maxnik. Conversely, if $G$ is maximal $2$-apex there are two cases, depending on $|V|$. If $n = |V| < 7$,
then $K_n$ is $2$-apex, so $G = K_n$. But, $K_n$ is also nIK and therefore maxnik. If $|V| \geq 7$, then $|E| = 5|V|-15$. Since
$G$ is 2-apex, it is nIK. Adding any edge $e$, we have $G+e$ with $5|V|-14$ edges. It follows that $G$ has a $K_7$ minor and is 
IK \cite{M, CMOPRW}. This shows that $G$ is maxnik.
\end{proof}

A similar result, with essentially the same proof, holds for maxnil.

\begin{thm}
An apex graph is maxnil if and only if it is maximal apex.
\end{thm}

\begin{thm} 
\label{thm:ord8}
For $|V| = n \leq 6$, $K_n$ is the only maxnik graph.
The only maxnik graphs for $n = 7$ and $8$ are the three $2$-apex graphs derived
from triangulations on five and six vertices.
\end{thm}

\begin{proof}
In \cite[Proposition 1.4]{M} it's shown that every nIK graph of order 8 or less is $2$-apex. So, the maxnik graphs are the 
maximal $2$-apex graphs. For $n \leq 6$, all graphs are $2$-apex, so $K_n$ is the only maximal knotless graph. For $n = 7$, the 
maximal $2$-apex graph is $K_7^-$, formed
by adding two vertices to the unique graph with a planar triangulation on five vertices, $K_5^-$.
The two maximal planar graphs on 8 vertices are 
formed by adding two vertices to the two triangulations on 
six vertices, the octahedron and a graph whose complement is a $3$-path. We
will call these graphs $K_8 - $3 disjoint edges and $K_8 - P_3$.
\end{proof}

Let $E_9$ (called $N_9$ in \cite{HNTY}) be the nIK nine vertex graph in the Heawood family.
Figure~\ref{fig:E9} in Section~\ref{sec:bounds} below shows a
knotless~\cite{M} embedding of $E_9$.

\begin{thm} The graph $E_9$ is maxnik.
\label{thm:E9}
\end{thm}

\begin{proof}
That $E_9$ is nIK is established in~\cite{M}.
Up to symmetry, there are two types of edges that may be added.
One type yields the graph $E_9+e$, shown to be IK (in fact minor minimal IK or MMIK) in \cite{GMN}. The other possible addition
yields a graph that has as a subgraph $F_9$ in the Heawood family. Kohara and Suzuki~\cite{KS} established that $F_9$ is MMIK.
\end{proof}

\begin{figure}[htb]
\begin{center}
\includegraphics[scale=0.5]{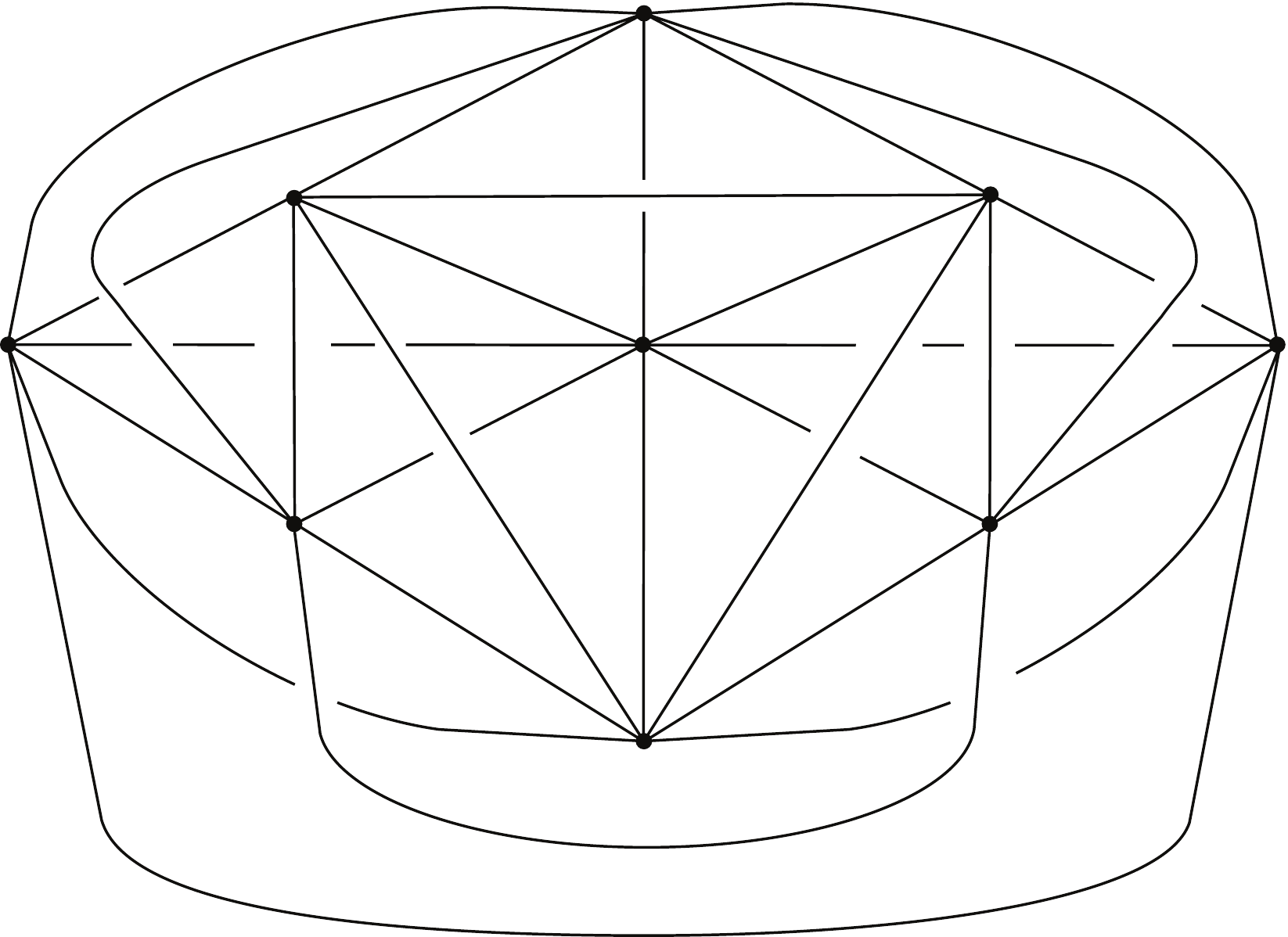}
\end{center}
\caption{A knotless embedding of $G_{9,29}$. 
\label{fig:G929}%
}
\end{figure}

\begin{thm}
There are seven maxnik graphs of order nine.
\end{thm}

\begin{proof}
The seven graphs are the five maximal $2$-apex graphs with 30 edges, $E_9$,  and the graph $G_{9,29}$, 
shown in Figure~\ref{fig:G929}. Note that $G_{9,29}$ is the complement of $K_1 \sqcup K_2 \sqcup C_6$.
Theorems~\ref{thm:twoapex} and \ref{thm:E9} show that six of these seven graphs are maxnik. To see that 
$G_{9,29}$ is as well, note that the embedding shown in Figure~\ref{fig:G929}, due to Ramin Naimi~\cite{N}, is knotless. 
Up to symmetry, there are two ways to add an edge to the graph. In either case, the new graph has a $K_7$ minor and is IK.

It remains to argue that no other graphs of order nine are maxnik. We know that order nine graphs with size 21 or less are
either IK, the graph $E_9$, or else $2$-apex, see \cite[Propositions 1.6 and 1.7]{M}. Using Theorem~\ref{thm:twoapex}, 
this completes the argument for graphs with $|E| \leq 21$. Suppose $G$ is maxnik of order nine with $|E| \geq 22$.
By Theorem~\ref{thm:mind}, we can assume $|E| \leq 30$.  If $G$ is $2$-apex, by Theorem~\ref{thm:twoapex}, it is one of the 
five maximal $2$-apex graphs. So, we can assume $G$ is not $2$-apex. The minor minimal not $2$-apex (MMN2A)
graphs through order nine
are classified in~\cite{MP}. With a few exceptions these graphs are also MMIK. If $G$ has an IK minor (including a MMIK minor)
it is IK and not maxnik. So, we can assume $G$ has as a minor a graph that is MMN2A, but not MMIK. 
There are three such graphs. One is $E_9$, the other two, $G_{26}$ and $G_{27}$, have 26 and 27 edges. 
In Theorem~\ref{thm:E9}, we showed that $E_9$ is maxnik. The other two are subgraphs of $G_{9,29}$. To complete 
the proof, we observe that any order nine graph that contains $G_{26}$ is either a subgraph of $G_{9,29}$ or else IK and similarly 
for $G_{27}$. In fact, for those that are IK, we can verify this by finding a MMIK minor, either in the $K_7$ or $K_{3,3,1,1}$ family,
or  else the graph $G_{9,28}$ described in~\cite{GMN}.
\end{proof}

\begin{thm} The only maxinik graph of size 20 is $K_7^-$. There are seven maxnik graphs with at most 20 edges.
\end{thm}

\begin{proof}
Work above establishes this through order nine. The seven maxnik graphs with at most 20 edges are the seven on seven 
or fewer vertices. Suppose $G$ of order ten or more and size 20 is maxinik. By 
\cite[Theorem 2.1]{M}, $G$ is $2$-apex and therefore maximal $2$-apex. But this means $|E| = 5 |V| - 15 \geq 35$, a contradiction.
\end{proof}

\begin{rem} A computer search suggests that $E_9$ is the only maxnik graph of size 21. The search makes use of the 92 
known MMIK graphs of size 22, see~\cite{FMMNN}.
\end{rem}

\section{Clique sums of maxnik graphs}
\label{sec:cliq}

Clique sums of maxnil graphs were studied in \cite{NPP}, and we will show similar, if weaker, versions in the case of maxnik graphs. These results are used in Section \ref{sec:bounds}.

\begin{lem}
\label{lem:NPP7}
For $t \leq 2$, the clique sum over $K_t$ of nIK graphs is nIK.
\end{lem}

\begin{proof}

Let $G_1$ and $G_2$ be nIK graphs, and let $\Gamma(G)$ denote the set of all cycles in $G$.  Let $G$ be the clique sum of $G_i$ over a clique of size $t$.  Let $f_i$ be an embedding of $G_i$ that contains no non-trivial knot.

Suppose $t=1$.  We may extend the $f_i$ to an embedding of $G$ by embedding $f_1(G_1)$ in 3-space with $z > 0$, and $f_2(G_2)$ with $z<0$.  $G = G_1 \cup_v G_2$, so by isotoping vertex $v$ from each $G_i$ to the plane $z=0$ and identifying them there, we have an embedding $f(G)$.  A closed cycle in $G$ must be contained in a single $G_i$, and hence given $c \in \Gamma(G)$, then $c \in \Gamma(G_i)$ for some $i$. As the embeddings $f_i(G_i)$ contain no nontrivial knot, $c$ must be the unknot, and hence $G$ is nIK.

Suppose $t=2$. We may extend the $f_i$ to an embedding of $G$ by embedding $f_1(G_1)$ in 3-space with $z > 0$, and $f_2(G_2)$ with $z<0$.  $G = G_1 \cup_e G_2$, so by shrinking the edge $e$ in each $G_i$ and then isotoping them to the plane $z=0$ and identifying them there, we have an embedding $f(G)$.  A closed cycle $c \in \Gamma(G)$ must either be an element of $\Gamma(G_i)$, or $c = c_1 \# c_2$, with $c_i \in \Gamma(G_i)$.  As the embeddings $f_i(G_i)$ contain no nontrivial knot, in the first case $c$ is the unknot, and in the second, it is the connect sum of unknots and hence unknotted.   Thus, $G$ is nIK.
\end{proof}

For $H_1, H_2, \ldots, H_k$ subgraphs of graph $G$, let $\langle H_1, H_2, \ldots, H_k \rangle_G$ denote the subgraph induced by the vertices of 
the subgraphs.

\begin{lem}
\label{lem:NPP9}%
Let $G$ be a maxnik graph with a vertex cut set $S = \{x,y\}$, and let $G_1, G_2, \ldots, G_r$ denote the connected components of 
$G \setminus S$. Then $xy \in E(G)$ and $\langle G_i, S \rangle_G$ is maxnik for all $1 \leq i \leq r$.
\end{lem}

\begin{proof}
As $G$ is 2-connected by Theorem \ref{thm:mind}, each of $x$ and $y$ has at least one neighbor in each $G_i$.  Suppose $xy \notin G$.  Form $G' = G + xy$ and let $G'_i = \langle G_i, S \rangle_{G'}$. For each $i$, edge $xy \in G'_i$. But $G'_i$ is a minor of $G$, as there exists $G_j$ with $i \neq j$ since $S$ is separating, and there exists a path from $x$ to $y$ in $G_j$ as $G_j$ is connected.   Thus in $\langle G_i, G_j, S \rangle_G$, we may contract $G_j$ to $x$ to obtain a graph isomorphic to $G'_i$.  Thus, $G'_i$ is nIK.  So, by Lemma \ref{lem:NPP7}, $G' = G'_1 \cup_{xy} G'_2 \cup_{xy} \ldots \cup_{xy} G'_r$ is nIK.  This contradicts the fact that $G$ is maxnik, and hence $xy \in E(G)$.  

Suppose that one or more of the $G_i$ are not maxnik.   Then add edges as needed to each $G_i$ to form graphs $H_i$ that are maxnik.   Then the graph $H = H_1 \cup_{xy} H_2 \cup_{xy} \ldots \cup_{xy} H_r$ is nIK by Lemma \ref{lem:NPP7} and contains $G$ as a subgraph.  As $G$ is maxnik, $G=H$ and hence $G_i = H_i$ for all $i$, so every $G_i$ is maxnik as well. 
\end{proof}

\begin{lem}
\label{lem:NPP10}%
Let $G_1$ and $G_2$ be maxnik graphs. Pick an edge in each $G_i$ and label it $e$. Then $G = G_1 \cup_{e} G_2$ is maxnik if $e$ is non-triangular in at least one $G_i$.
\end{lem}

\begin{proof}
Suppose that $e$ is non-triangular in $G_1$ and has endpoints $x,y$.  Add an edge $ab$ to the graph $G$.  The graph $G$ is nIK by Lemma \ref{lem:NPP7}. If both $a, b \in G_i$ for some $i$, then $G + ab$ is IK, as the $G_i$ are each maxnik.  Thus, we may assume that $a \in G_1$ and $b \in G_2$.  The edge $e$ is non-triangular in $G_1$, so vertex $a$ is not adjacent to both endpoints of $e$.   We may assume that $a$ is not adjacent to $x$.  As $G_2$ is connected, we construct a minor $G'$ of $G +ab$ by contracting the whole of $G_2$ to vertex $x$.  Note that as $b \in G_2$, we have the edge $ax$ in $G'$, and in fact $G' = G_1 +ax$.   As $G_1$ is maxnik, $G'$ is IK and so is $G + ab$.   Thus, $G$ is maxnik.
\end{proof}

\begin{lem}%
\label{lem:New}%
For $i = 1,2$, let $G_i$ be maxnik, containing a $3$-cycle $C_i$, and admitting to a knotless embedding such that $C_i$ bounds a disk whose interior
is disjoint from the graph. Then the clique sum $G$ over $K_3$ formed by identifying $C_1$ and $C_2$ is nIK. Moreover, $G$ is maxnik if $C_i$ is not 
part of a $K_4$ in at least one $G_i$.
\end{lem}

\begin{proof} 
Let $f_i$ be the knotless embedding of $G_i$.  Embed the $f_i(G_i)$ so that they are separated by a plane.   We may then extend this to an embedding $f(G)$ by isotoping the $C_i$ to the separating plane and identifying them there.  

 Let $\Gamma(G)$ denote the set of all cycles in $G$. As the cycles $C_i$ bound a disk in $f(G)$, if a closed cycle $c \in \Gamma(G)$ is not contained in one of the $f_i(G_i)$, then  $c = c_1 \# c_2$, with $c_i \in \Gamma(G_i)$.  As the embeddings $f_i(G_i)$ contain no nontrivial knot, in the first case $c$ is the unknot, and in the second, it is the connect sum of unknots and hence unknotted.   Thus, $G$ is nIK.
 
 Suppose $C_1$ is not contained in a 4-clique  in $G_1$. We will show $G + ab$ is IK, and hence $G$ is maxnik.  As the $G_i$ are maxnik, we may assume that $a \in G_1$ and $b \in G_2$, as otherwise $G + ab$ is IK.  As $C_1$ is not contained in a 4-clique in $G_1$, there exists a vertex $x$ in $C_1$ that is not adjacent to $a$.  As $G_2$ is connected, there is a path from $b$ to $x$.  Contract $G_2$ to $a$.  This graph contains $G_1 + ax$ as a minor, and hence is IK, as $G_1$ is maxnik and does not contain edge $ax$.   Thus, $G$ is maxnik.  
 
\end{proof}

\section{Bounds on maximal knotless graphs}
\label{sec:bounds}

We now consider maximal knotless graphs in general and establish bounds on the possible number of edges, and the maximal and minimal degrees.   We first show a lemma that will be useful for establishing a lower bound. A similar result holds for maximal linkless graphs as well.

\begin{lem}
Suppose $G$ is maxnik and contains a vertex $v$ of degree 3. Then all neighbors of $v$ are adjacent to each other. 
\label{lem:deg3_nghs}
\end{lem}

\begin{proof}

Label the neighbors of $v$ as $x_1,x_2,x_3$. Let $E_v = \{x_1x_2, x_1x_3, x_2x_3 \}$ and $E = E(G)$.
Delete the edges in $E\cap E_v$ to form $G_Y = G \setminus (E \cap E_v)$.  Then add back all the edges of $E_v$ to form $G' = G_Y + E_v$.  We will show $G = G'$. 

As $G$ is maxnik, $G_Y$ has an embedding $f$ with no nontrivial knot.  We may extend $f$ to an embedding of $G'$ by embedding each edge $x_ix_j$ so that the 3-cycle $x_ivx_j$ bounds a disk.

Let $\Gamma(G)$ denote the set of all cycles in the graph $G$. Suppose $c$ is a cycle in $\Gamma(G')$.   If $c$ does not contain one or more edges $x_ix_j$, then $c \in \Gamma(G_Y)$, and hence is a trivial cycle in $f(G')$.   Suppose that $c$ does contain one or more edges $x_ix_j$.   There are three possibilities: 
$c$ is a $3$-cycle  $x_i v x_j$ and bounds a disk; $c$ includes a path of the form $x_i, x_j, v, x_k$ with $\{i,j,k \} = \{1,2,3\}$; or $c$ does not include the vertex $v$.
In the first case $c$ is trivial as it bounds a disk. If 
$c$ does not contain $v$, then, since the cycles $x_ivx_j$ bound disks, $c$ is isotopic to $c' \in \Gamma(G_Y)$ and hence trivial.
Similarly, if $c$ includes a path $x_i, x_j, v, x_k$, using the disk $x_i v x_j$, we can isotope the path to $x_i, v, x_k$ to make $c$ isotopic
to $c' \in \Gamma(G_Y)$ and hence trivial. 

Thus, $G'$ has an embedding with no non-trivial knot.   As $G$ is maxnik, $G$ cannot be a proper subgraph of $G'$, and hence $G=G'$.
\end{proof}

\begin{thm} 
\label{thm:2vbound}%
If $G$ is maxnik with  $|V| \geq 5$, then $|E| \geq \frac74|V|$.
\end{thm}

\begin{proof} 
By Theorem~\ref{thm:ord8}, $K_5$ is the only maxnik graph with order five and it satisfies the conclusion of the theorem.

Suppose $H$ has the least number of vertices among counterexamples to the theorem. We will consider a vertex $v$ of minimal degree in $H$.  If $\deg(v) \geq 4$, then $H$ has $|E| \geq 2|V|$ and hence is not a counterexample, so $\deg(v) \leq 3$. 
By Theorem~\ref{thm:mind}, $\deg(v) \geq 2$, so we need only consider $v$ of degree 2 or 3.

Suppose $\deg(v) = 2$.  We will argue that $H' = H \setminus v$ is also maxnik with $|E'| < \frac74 |V'|$, contradicting our assumption that 
$H$ was a minimal counterexample. Let $N(v) = \{w,x\}$ and note that $wx \in E(H)$. Otherwise, in an unknotted
embedding of $H$, we could add the edge $wx$ so that the $3$-cycle $vwx$ bounds a disk. This will not introduce 
a knot into the embedding and contradicts the maximality of $H$.

As a subgraph of $H$, $H'$ is nIK. Suppose it is not maxnik because there is an edge $ab$ so that $H'+ab$ 
remains nIK. In a knotless embedding of $H'+ab$, we can add the vertex $v$ and its two
edges so the $3$-cycle $vwx$ bounds a disk. This will not introduce a knot into the embedding and shows
that $H+ab$ is also nIK, contradicting the maximality of $H$. Thus, no such graph $H$ with a vertex of degree 2 can exist.

So, we may assume that $\deg(v)=3$.  Here we cannot apply the techniques of \cite{A}, as $Y \nabla$ moves do not preserve intrinsic knotting \cite{FN}.   However, Lemma \ref{lem:deg3_nghs} allows us to show the average degree of $H$ is actually at least 3.5, and hence $H$ is not a counterexample.

Divide the vertices of $H$ into 3 sets:  $A = \{$vertices of degree 3$\}$, $B=\{$vertices of degree $>3$ that are neighbors of vertices in $A\}$, and $C = \{$all other vertices of $H \}$.  Form the graph $H' = H \setminus C$.   All vertices in $C$ have degree 4 or greater, so it suffices to show that the vertices in each connected component of $H'$ have average degree 3.5 or higher. 

A vertex $a_{i1}$ of degree 3 has three neighbors, label them $b_{i1}, b_{i2}, a_{i2}$, where $a_{i2}$ is a neighbor of minimal degree.  If $\deg(a_{i2})=3$, we continue.  If not, delete all edges incident on $a_{i2}$ except those between $a_{i2}$ and $\{a_{i1}, b_{i1}, b_{i2}\}$.  This creates a subgraph of $H'$ with strictly fewer edges; we will abuse notation and continue to call it $H'$.  Vertex $a_{i2}$ now has degree 3 in $H'$, and we move it to set $A$.   

If $a_{i2}$ had degree greater than 3 in $H$, then, since it has the minimal degree among the neighbors of $a_{i1}$, $\deg(b_{ij}) \geq 4$ and $b_{i1}, b_{i2} \in B$.
If $\deg(a_{i2})=3$ in $H$, vertices $a_{ij}$ are adjacent only to each other and the $b_{ij}$.  If either of the $b_{ij}$ have degree 3 in $H$, then $H$ can be disconnected by deleting the other $b_{ij}$.  This is a contradiction as $H$ is maxnik and must be 2-connected by Theorem \ref{thm:mind}.  Thus, the $b_{ij}$ are in $B$.

Consider the connected component of $v$ in  $H'$, call it $H'_1$.  We will calculate the total degree of the vertices in $H'_1$ and divide by the number of vertices. Suppose there are $n$ vertices from set $A$ and $m$ vertices from set $B$ in $H'_1$ for a total of $n+m$ vertices.  Each vertex from set $A$ has degree 3, so the contribution to total degree from set $A$ is $3n$.  Each vertex in $A$ is adjacent to exactly 2 of the $b_{ij}$, so the total degree contribution for set $B$ is at least $2n$ from edges to set $A$.  Further, $H'_1$ is connected.  As $a_{ij}$ is only adjacent to $b_{i'j'}$ if $i=i'$, there must be at least $m-1$ edges between the $b_{ij}$, which adds $2(m-1)$ to the total degree.  This gives an average degree of $\frac{5n+2m-2}{n+m}$ in $H'_1$. However, $H$ is 2-connected by Theorem \ref{thm:mind}, so there must be at least 2 edges from $H'_1$ to its complement in $H$. So within $H$, these vertices must have average degree greater than or equal to $\frac{5n+2m}{n+m}$.  Note that $ 2 \leq m \leq n$, and $\frac{5n+2m}{n+m}$ attains its minimum at $m = n$. The minimum is $\frac72$, and hence $H$ must have $|E| \geq \frac74 |V|$.            
\end{proof}

\begin{thm}
\label{thm:NPP5}%
There exist maxnik graphs with  $|E| < \frac52|V|$ edges for arbitrarily large $|V|$.
\end{thm}

\begin{proof}
Let $e$ be an edge of $E_9$ connecting a degree 4 vertex to one of degree 5. 
Edge $e$ is non-triangular and there are five other edges symmetric to it.
Using Lemma~\ref{lem:NPP10}, take $k$ copies of $E_9$ glued along edge $e$. The resulting graph
has $7k+2$ vertices and $20k+1$ edges.  Gluing on five $K_3$ graphs in each $E_9$ on the other non-triangular edges gives an additional $5k$ 
vertices and $10k$ edges.  So, for each $k \geq 1$, we have a graph $G$ with $n= 12k+2$ vertices and $m = 30k+1$ edges.  
Then $m = 30(n-2)/12 + 1 = \frac52 n - 5$.
\end{proof}

These two theorems suggest the following question. In Table~\ref{tbl:VE} we give the least ratios through order nine.

\begin{que}
What is the minimal number of edges for a maxnik graph of $n$ vertices? 
\end{que}

\begin{table}
\begin{center}
\begin{tabular}{c|ccccccccc}
$|V|$ & 1 & 2 & 3 & 4 & 5 & 6 & 7 & 8 & 9 \\ \hline
min($|E|/|V|$) & 0 & 1/2 & 1 & 3/2 & 2 & 5/2 & 20/7 & 25/8 & 21/9
\end{tabular}
\end{center}
\caption{The least ratios of size to order for maxnik graphs through order nine.
\label{tbl:VE}%
}
\end{table}

For maximal planar graphs, $|E| = 3|V| - 6$. Similarly, maximal $k$-apex graphs have a fixed number of edges depending on $|V|$. In contrast, as with maximal linkless graphs, the number of edges in a maxnik graph can vary. In fact, with the exception of $|E| = 22$, for any $|E| \geq 20$, there exists a maxnik graph of that size.

\begin{figure}[htb]
\begin{center}
\includegraphics[scale=0.5]{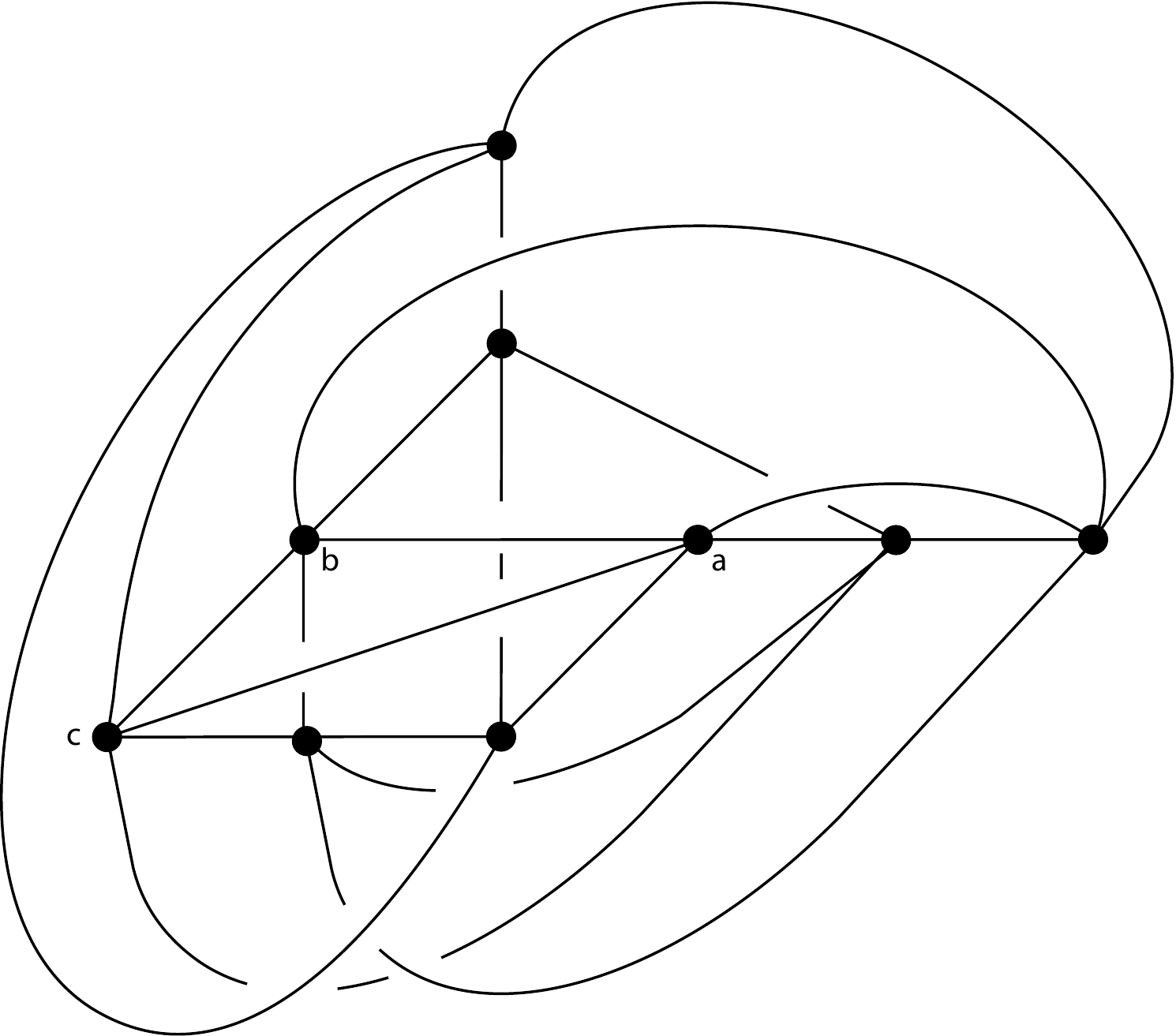}
\end{center}
\caption{A knotless embedding of $E_9$. 
\label{fig:E9}%
}
\end{figure}

\begin{thm}
Let $n \geq 20$ and $n \neq 22$. Then there exists a maxnik graph with $|E|  = n$.
\label{thm:size}
\end{thm}

\begin{proof}
The graph $K_7^{-}$ is a maxnik of size 20 by Theorem \ref{thm:ord8}. The graph $E_9$ has a knotless embedding where the 3-cycle $abc$ bounds a disk~\cite{M}, 
shown in Figure~\ref{fig:E9}.  As no vertex in $E_9$ is adjacent to all three of these vertices, we may use Lemma~\ref{lem:New} to construct a maxnik graph of size 24 by taking a clique sum over $K_3$ of 
$E_9$ and $K_4$. So, we may assume $n \geq 21$ and $n \not\in \{22,24\}$.

The graph $E_9$ has size 21 and 6 non-triangular edges. Let $G_i$ denote the maxnik graph obtained from $i$ copies of $E_9$ by gluing along non-triangular edges.  

Note that $|E(G_{i+1})| - |E(G_i)| = 20$, and that $G_i$ contains at least 6 non-triangular edges for any $i$.  We will now work by induction.  Suppose that maxnik graphs exist for size $n < |E(G_i)|$ and for size $|E(G_i)| +1$ and size $|E(G_i)| +3$.   Then it suffices to show that there exist maxnik graphs of size $|E(G_i)| +k$ for $4 \leq k \leq 19$ and $k \in \{0, 2, 21, 23\}$.  

Clearly a maxnik graph of size $|E(G_i)|+0$ exists, as $G_i$ is maxnik.  We may form a new maxnik graph from $G_i$ by gluing a copy of $K_m$  (for $3\leq m \leq 6$) along a non-triangular edge of $G_i$ by Lemma \ref{lem:NPP10}.  As $G_i$ has at least 6 non-triangular edges, we can glue on up to 6 such graphs, each adding ${m \choose 2} -1$ edges.  Thus, to prove the result we need only to be able to form the desired values of $k$ using six or fewer addends from the set $\{2,5,9,14\}$.  This is clearly possible.  


In the base case $i=1$, we have a maxnik graph of size $|E(G_1)|=21$, and we excluded graphs of size 22 and 24  ($|E(G_1)| +1$ and $|E(G_1)| +3$) above.  Thus we may form maxnik graphs of size $|E(G_1)| + k$ for the $k$ of interest as before.   
\end{proof}

\begin{rem} A computer search shows there are no size 22 maxnik graphs. 
Our strategy is based on the classification through size 22 of the obstructions to $2$-apex in \cite{MP}.
Let's call such graphs MMN2A (minor minimal not 2-apex). All but eight of the graphs in the classification
are MMIK. Two exceptions are $4$-regular of order 11, the other six are in the Heawood family.

A maximal $2$-apex graph has $5n-15$ edges where $n$ is the number
of vertices. By Theorem~\ref{thm:twoapex} a maxnik graph $G$ of size 22 is not $2$-apex and therefore has a MMN2A minor.
Since $G$ is nIK, it must have one of the eight exceptions as a minor. Using a computer, we verified that no size 22 expansion 
of any of these eight graphs is maxnik.
\end{rem}

Theorem \ref{thm:size} implies that there are maxnik graphs of nearly every size.  Note that there are maxnik graphs of any order, as there exist maximal 2-apex graphs of any order and by Theorem \ref{thm:twoapex} these graphs are maxnik.

We have considered the minimal number and the possible number of edges in a maxnik graph. We now consider other aspects of maxnik graphs' structure, in particular, the maximal and minimal degree. Since $\Delta(G) = |V|-1$ for maximal $2$-apex graphs, there are maxnik graphs with arbitrarily large $\Delta(G)$.

\begin{prop} The complete graph $K_3$ is the only maxnik graph with maximal degree two.
\end{prop}

\begin{proof} Suppose $G$ is maxnik with $\Delta(G) = 2$. Then $|G| \geq 3$ and,
by Theorem~\ref{thm:mind}, $\delta(G) = 2$ and $G$ is connected.
So, $G$ is a cycle. Now, a cycle is planar, hence $2$-apex, and by Theorem~\ref{thm:twoapex}, 
$G$ is maximal $2$-apex. However, 
a cycle is not maximal $2$-apex unless it is $K_3$. 
\end{proof}

Note that Lemma \ref{lem:deg3_nghs} has the following two immediate corollaries.

\begin{cor}
If a graph $G$ is maxnik and has $\Delta(G)=3$, then $G$ is 3-regular.
\end{cor}

\begin{cor}
If a graph $G$ is maxnik and 3-regular, then $G=K_4$.
\end{cor}

These results motivate the following question.

\begin{que}
Do there exist regular maxnik graphs other than $K_n$ with $n<7$?
\end{que}

A maximal $2$-apex graph will have $\Delta(G) = |V|-1$ and $\delta(G) \leq 7$, so if there is such a regular maxnik graph with $|V|  \geq 7$, it is not $2$-apex. 
However, through order nine, our two examples of maxnik non $2$-apex graphs are both close to regular, having
$\Delta(G) - \delta(G) \leq 2$.  This suggests the answer to our question is likely yes.

For $\delta(G)$, Theorem~\ref{thm:mind}
gives a lower bound of two that is realized by the infinite family of Theorem~\ref{thm:NPP5}.
On the other hand, by starting with a planar triangulation
of minimum degree five, we can construct graphs with $\delta(G) = 7$ that are maximal $2$-apex, and hence maxnik. At the same time, 
since a graph with $|E| \geq 5|V| - 14$ has a $K_7$ minor and is IK~\cite{Md,CMOPRW}, a maxnik graph must have 
$\delta(G) \leq 9$. 
It seems likely that there are examples that realize this upper bound on $\delta(G)$.  Table~\ref{tbl:deg} records the 
range of degrees for maxnik graphs through order nine.

\begin{table}
\begin{center}
\begin{tabular}{c|ccccccccc}
$|V|$ & 1 & 2 & 3 & 4 & 5 & 6 & 7 & 8 & 9 \\ \hline
$\delta(G)$ & 0 & 1 & 2 & 3 & 4 & 5 & 5 & 5 or 6 & 4 to 7 \\
$\Delta(G)$ & 0 & 1 & 2 & 3 & 4 & 5 & 6 & 7 & 5 to 8 
\end{tabular} 
\end{center}
\caption{Maximal and minimial degrees of maxnik graphs through order nine.
\label{tbl:deg}%
}
\end{table}

\section{Prime and Composite Maxnik Graphs}
\label{sec:prime}

We will call a graph {\em composite} if it is the clique sum of two graphs.  Otherwise it is {\em prime}. These terms are analogous to knots, where a knot is composite if is the connected sum of two non-trivial knots, and prime otherwise. In this section, we
classify the maxnik graphs described earlier in this paper as prime and composite. We remark that it may be of interest to study other instances of prime graphs, 
for example, prime maximal planar or prime maxnil.

The infinite families of maxnik graphs constructed in Section \ref{sec:bounds} are all composite, as they are clique sums of smaller maxnik graphs. 

Note that $K_n$ is prime, so all maxnik graphs of order 6 or less are prime.

\begin{figure}[htb]
\begin{center}
\includegraphics[scale=0.5]{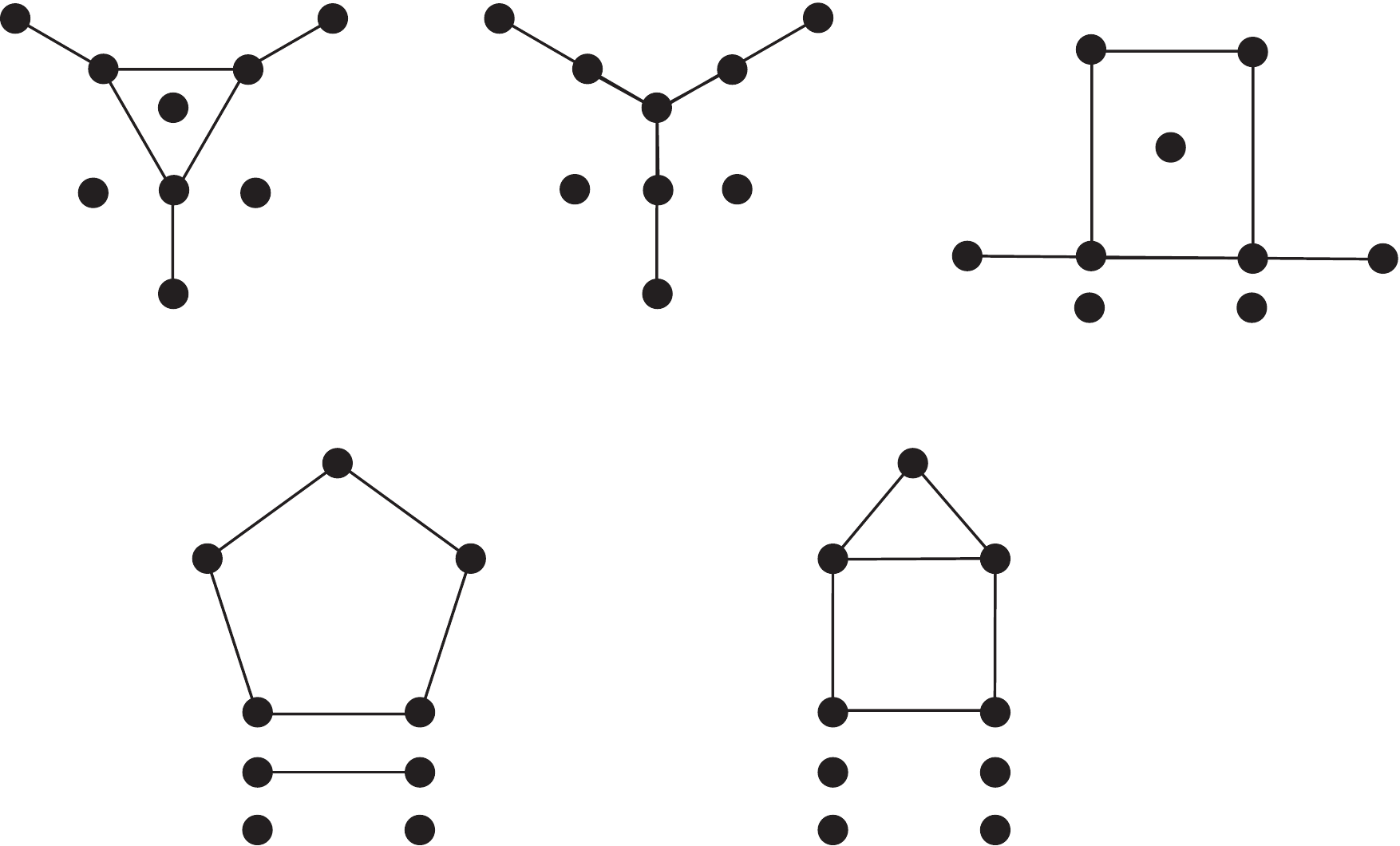}
\end{center}
\caption{Complements of the maximal 2-apex graphs of order nine. Top row, L to R: Big-Y, Long-Y, Hat; Bottom row: Pentagon-bar and House.
\label{fig:tri}%
}
\end{figure}

\begin{prop}
The following maxnik graphs are composite: $K_7^-$, $K_8 - P_3$, and four of the five maximal 2-apex graphs on 9 vertices, specifically Big-Y, Long-Y, Hat and House.
\end{prop}

\begin{proof}

The graph $K_7^-$ is formed from two copies of $K_6$ summed over a 5-clique.

The graph $K_8 - P_3$ is formed from  $K_7^-$ clique sum $K_6$ over a 5-clique, where the 5-clique contains exactly one endpoint of the missing edge.

Big-Y is formed from $K_8 - P_3$ clique sum $K_6$ over a 5-clique, where the 5-clique contains both of the terminal vertices of the 3-path.

Long-Y is formed from $K_8 - $ 3 disjoint edges clique sum $K_6$ over a 5-clique.

Hat is formed from $K_8-P_3$ clique sum $K_6$ over a 5-clique, where the 5-clique contains one terminal vertex and one (non-adjacent) interior vertex of the 3-path.

House is formed from $K_8-P_3$ clique sum $K_6$ over a 5-clique, where the 5-clique contains one interior vertex of the 3-path.

\end{proof}

\begin{lem}
If $G^c$ is of the form $K_2 \coprod H$, then either $G$ is prime, or $G$ is the clique sum of two copies of $K_n$ over an $n-1$ clique.
\label{lem:comp_K2}
\end{lem}

\begin{proof}
Call the two vertices of the $K_2$ in $G^c$  $v_1$ and $v_2$.  Suppose that $G$ is a clique sum of $G_1$ and $G_2$ over a clique $C$.  We cannot have both $v_1$ and $v_2$ in $C$, as edge $v_1v_2$ is in $G^c$.  Without loss of generality, we may assume that $v_1$ is in $G_1 \setminus C$.  So, in $G^c$, $v_1$ must be adjacent to every vertex of $G_2 \setminus C$.  Thus $G_2 \setminus C$ is $v_2$.  As the only neighbor of $v_1$ in $G^c$ is $v_2$, $v_1$ is adjacent to every vertex in $C$.  Similarly for $v_2$.  Thus if $G$ is composite, it is the clique sum of $K_n$ and $K_n$ over an $n-1$ clique.
\end{proof}

\begin{cor}
The following maxnik graphs are prime: Pentagon-bar, $G_{9,29}$ and $K_8 -$ 3 disjoint edges.
\end{cor}

\begin{proof}
Each of these graphs has a complement of the form $K_2 \coprod H$.  As these graphs are not of the form $K_n - $ a single edge, they are prime by Lemma \ref{lem:comp_K2}.
\end{proof}

Note that if $G$ is a clique sum over a $t$-clique, it is not $(t+1)$-connected.

\begin{prop}
 The maxnik graph $E_9$ is prime.
\end{prop}

\begin{proof}
The largest clique in $E_9$ is a 3-clique, but $E_9$ is 4-connected and hence must be prime.
\end{proof}

\begin{lem}
If $G = H*K_2$, and $G$ is 2-apex, then $G$ is prime maxnik if and only if $H$ is prime maximal planar.
\end{lem}

\begin{proof}
As $G$ is 2-apex, it is maxnik if and only if it is maximal 2-apex, and $G$ is maximal 2-apex if and only if $H$ is maximal planar.  

If $H$ is composite, then $H$ is the clique sum of $H_1$ and $H_2$ over a $t$-clique. So $G$ is the clique sum of $H_1*K_2$ and $H_2*K_2$ over a $t+2$-clique, and hence $G$ is composite.

As $G$ is maxnik, it must be 2-connected. Hence if $G$ is composite, it must be $G_1$ clique sum $G_2$ over a $t$-clique $C$, with $t \geq 2$.  Label two of the vertices in $C$ as $v_1, v_2$.  Then $H$ is the clique sum of $G_1 \setminus \{v_1, v_2\}$ and $G_2\setminus \{v_1, v_2\}$ over $C\setminus \{v_1, v_2\}$, and thus composite.  
\end{proof}

\begin{cor}
There exist prime maxnik graphs of arbitrarily large size, and of any order $\geq 8$.
\end{cor}

\begin{proof}
The octahedron graph is max planar and 4-connected.   The largest clique it contains is a 3-clique, so it is prime.  New triangulations formed by repeated subdivision of a single edge are 4-connected and maximal planar, but have no 4-clique, hence are prime as well.   Thus all of these graphs give prime maxnik examples when joined with $K_2$.
\end{proof}

We remark that the construction of this family of graphs is similar to the maxnil families with $3n-3$ edges due to J{\o}rgensen~\cite{J} and $3n-5$ edges  
due to Naimi, Pavelescu, and Pavelescu~\cite{NPP}.

\end{document}